\documentclass[11pt]{amsart} 

\usepackage[english]{babel}
\usepackage[utf8]{inputenc}
\usepackage[T1]{fontenc}

\usepackage{amsmath}
\usepackage{amssymb}
\usepackage{amsfonts}
\usepackage{amssymb}
\usepackage{amsthm}
\usepackage{amscd}
\usepackage{xcolor}

\usepackage[hidelinks]{hyperref}
\usepackage{geometry}

\newcommand{\setsymbol}[1]{\ensuremath{\mathbb{#1}}}%
\newcommand{\R}{\setsymbol{R}}%
\newcommand{\Sphere}{\setsymbol{S}}
\newcommand{\Torus}{\setsymbol{T}}%

\newtheorem{definition}{Definition}[section]
\newtheorem{lemma}[definition]{Lemma}
\newtheorem{theorem}[definition]{Theorem}
\newtheorem{proposition}[definition]{Proposition}
\newtheorem{corollary}[definition]{Corollary}
\newtheorem{example}[definition]{Example}
\newtheorem{remark}[definition]{Remark}

\DeclareMathOperator{\Rm}{Rm}
\DeclareMathOperator{\Ric}{Ric}
\DeclareMathOperator{\scal}{R}

\title{A generalization of Geroch's conjecture}

\author{Simon Brendle}
\address{Department of Mathematics \\ Columbia University \\ New York, NY, 10027}
\email{simon.brendle@columbia.edu}

\author{Sven Hirsch}
\address{Department of Mathematics \\ Duke University \\ Durham, NC, 27708}
\email{sven.hirsch@duke.edu}

\author{Florian Johne}
\address{Department of Mathematics \\ Columbia University \\ New York, NY, 10027}
\email{johne@math.columbia.edu}

\begin{document}

\begin{abstract}
The Theorem of Bonnet--Myers implies that manifolds
with topology 
$M^{n-1} \times \Sphere^1$ do not admit a metric of positive Ricci curvature,
while the resolution of Geroch's conjecture implies
that the torus $\Torus^n$ does not admit a metric of positive scalar curvature.
In this work we introduce a new notion
of curvature interpolating between Ricci and scalar curvature
(so called $m$-intermediate curvature),
and use stable weighted slicings to show that for $n \leq 7$ and $1 \leq m \leq n-1$ 
the manifolds $N^n = M^{n-m} \times \Torus^m$ do not admit a metric 
of positive $m$-intermediate curvature.
\end{abstract}

\maketitle

\section{Introduction}

Closed manifolds with positive Ricci curvature have finite fundamental group
due to the Theorem of Bonnet--Myers, in particular manifolds
of topological type $N^n = M^{n-1} \times \Sphere^1$
do not admit a metric of positive Ricci curvature.
A different proof (at least in dimension $n \leq 7$)
can be obtained by minimizing area in a homology class
and using the stability inequality with test function $f = 1$. 

On the other hand, a conjecture of Geroch asks whether the torus $\Torus^n$
does admit a metric of positive scalar curvature.
This conjecture was resolved by R.~Schoen und S.-T.~Yau 
for $3 \leq n \leq 7$ by using minimal hypersurfaces \cite{Schon-Yau:1979:Structure_PositiveScalar},
and by M.~Gromov and H.-B.~Lawson by using spinors for all dimensions \cite{GromovLawson}.
The non-existence result for metrics of positive scalar curvature
was extended to closed $n$-dimensional aspherical manifolds
for $n \in \{4, 5\}$ independently by O.~Chodosh and C.~Li \cite{ChodoshLi} 
and by M.~Gromov \cite{Gromov}.
For a more detailed overview on topological obstructions to positive scalar curvature we refer to the recent survey by O.~Chodosh and C.~Li \cite{ChodoshLiSurvey}. 

The above obstruction for positive Ricci curvature
and positive scalar curvature raise the following question:
What kind of curvature obstructions can be found for manifolds
of topological type $N^n = M^{n-m} \times \Torus^m$?
This is an interesting question even for the case $N^4 = \Sphere^2 \times \Torus^2$.

To investigate this question we define a family of curvature conditions (for $1 \leq m \leq n-1$)
reducing to Ricci curvature for $m = 1$
and to scalar curvature for $m = n-1$ as follows:
\begin{definition}[Positive $m$-intermediate curvature] \ \\
 Suppose $(N^n,g)$ is a Riemannian manifold.
 For given orthonormal vectors $\{e_1, \dots, e_m\}$ at the point $p \in N$
 extend them to an orthonormal basis $\{e_1, \dots, e_n\}$
 of $T_p M$.
 The $m$-intermediate curvature $\mathcal{C}_m$ of the orthornormal
 vectors $\{e_1, \dots, e_m\}$ is defined by
 \[
  \mathcal{C}_m(e_1, \dots, e_m) := \sum_{p=1}^m\sum_{q=p+1}^n 
  \Rm_N(e_p, e_q, e_p, e_q).
 \] 
 We say that $(N^n,g)$ has positive $m$-intermediate curvature at $p \in N$,
 if we have $\mathcal{C}_m(e_1, \dots, e_m) > 0$
 for any choice of orthornormal vectors $\{e_1, \dots, e_m\}$.
Moreover, we say that the manifold $(N^n,g)$ has positive $m$-intermediate curvature,
 if it has positive $m$-intermediate curvature for all $p \in M$.
\end{definition}

The product manifold $N^n = \Sphere^{n-m} \times \Torus^m$
(with $1 \leq m \leq n-2$)
with the standard metric on both factors has positive $(m+1)$-intermediate curvature, and nonnegative $m$-intermediate curvature.

We observe that the condition of positive $m$-intermediate curvature defines
a non-empty, open, $O(n)$-invariant convex cone in the space of algebraic curvature
tensors for $1 \leq m \leq n-1$. Moreover, under the conditions $2 \leq m \leq n-1$
and $n+2-m \leq k \leq n$
the curvature tensor of $\Sphere^{k-1} \times \R^{n-k+1}$
is contained in this open cone.
The general surgery result due to S.~Hoelzel \cite[Theorem A]{Hoelzel} then
implies that positive $m$-intermediate curvature is preserved
under surgeries of codimension at least $n+2-m$.

\begin{remark}[Connection to other notions of curvature] \ \\
(i) The quantity $ \mathcal{C}_m(e_1, \dots, e_m)$ is a sum of sectional curvatures of planes containing at least one of the vectors $e_1,\hdots,e_m$. In particular, positive $m$-intermediate curvature is a weaker condition than positive sectional curvature.

(ii) A manifold with positive $m$-intermediate curvature has positive scalar curvature. Indeed, the sum $\sum_{1 \leq p_1 < \dots < p_m \leq n} \mathcal{C}_m(e_{p_1}, \dots,e_{p_m})$ is equal to the scalar curvature, up to a factor. 

(iii) There is a connection to
the notion of $(m,n)$-intermediate scalar curvature
introduced (as $m$-curvature) into the literature by M.-L.~Labbi \cite{Labbi:1997:pCurvature}
and also studied by M.~Burkemper, C.~Searle and M.~Walsh \cite{BSW}. More precisely, the $(m,n)$-intermediate scalar curvature defined by
\begin{align*}
    s_{m,n}(e_1, \dots, e_m) =
    \sum_{p=m+1}^n \sum_{q=m+1}^n
    \Rm(e_p,e_q,e_p,e_q)
\end{align*}
satisfies the relation
\begin{align*}
    s_{m,n}(e_1, \dots, e_m)+2 \, \mathcal{C}_m(e_1, \dots, e_m)= \scal.
\end{align*}
In particular, 
$\mathcal{C}_m(e_1, \dots, e_m)$ depends only
on the span of $\{e_1, \dots, e_m\}$, hence the $m$-intermediate curvature $\mathcal{C}_m$
can be regarded as a scalar function on the Grassmannian.

(iv) For $m=n-1$, we obtain $s_{n-1,n}(e_1,\hdots,e_{n-1}) = 0$ and $2 \, \mathcal{C}_{n-1}(e_1,\hdots,e_{n-1}) = \scal$. Hence, the intermediate curvature reduces to the scalar curvature in this case.
\label{remark:ConnectionToOtherConditions}
\end{remark}

The case $m = 2$ (also called bi-Ricci curvature) was
studied by Y.~Shen and R.~Ye in \cite{Shen-Ye:1996:PositiveBiRicci, Shen-Ye:1997:PositiveBiRicci}.
They proved diameter estimates for stable minimal submanifolds
in manifolds of positive bi-Ricci curvature and an estimate
on the homology radius.

Our first main theorem concerns obstructions for the existence
of metrics of positive $m$-intermediate curvature. To that end, we consider a notion of stable weighted slicing. Our definition closely resembles the notion of minimal $k$-slicings
by R.~Schoen and S.-T.~Yau \cite{Schoen-Yau:2017:Structure_PositiveScalar_HigherDimensions}.

\begin{definition}[Stable weighted slicing of order $m$] \ \\
Suppose $1 \leq m \leq n-1$ and let
$(N^n,g)$ be an orientable Riemannian manifold of dimension $\dim N = n$. A stable weighted slicing of order $m$ consists of a collection of orientable and smooth submanifolds $\Sigma_k$, $0 \leq k \leq m$, and a collection of positive functions $\rho_k \in C^\infty(\Sigma_k)$ satisfying the following conditions: 
\begin{itemize}
\item $\Sigma_0 = N$ and $\rho_0 = 1$.
\item For each $1 \leq k \leq m$, $\Sigma_k$ is an embedded two-sided hypersurface in $\Sigma_{k-1}$. Moreover, $\Sigma_k$ is a stable critical point of the $\rho_{k-1}$-weighted area
 \[
  \mathcal{H}^{n-k}_{\rho_{k-1}}(\Sigma) = \int_\Sigma \rho_{k-1} \, d\mu
 \]
in the class of hypersurfaces $\Sigma \subset \Sigma_{k-1}$.
\item For each $1 \leq k \leq m$, the function $\frac{\rho_k}{\rho_{k-1}|_{\Sigma_k}} \in C^\infty(\Sigma_k)$ is a first eigenfunction of the stability operator associated with the $\rho_{k-1}$-weighted area. 
\end{itemize}
 \label{Definition: stable weighted slicing}
\end{definition}
Observe that we use the first eigenfunction of the Jacobi operator of weighted area, while in \cite[p.~7]{Schoen-Yau:2017:Structure_PositiveScalar_HigherDimensions} a perturbed version of the weighted
stability operator (denoted by $Q_k$)
is used.

It is a classical theorem that manifolds with positive Ricci curvature do not admit stable minimal hypersurfaces. Our first theorem shows that manifolds with $m$-intermediate curvature do not allow stable weighted slicings of order $m$.

\begin{theorem}[$m$-intermediate curvature and stable weighted slicings] \ \\
Assume that $1 \leq m \leq n-1$ and $n(m-2) \leq m^2 - 2$. Suppose $(N^n,g)$ is a closed and orientable Riemannian manifold with positive $m$-intermediate curvature. Then $N$ does not admit a stable 
 weighted slicing 
 \[
 \Sigma_m \subset \dots \subset \Sigma_1 \subset \Sigma_0 = N^n
 \]
 of order $m$. 
 \label{theorem:SmoothSlicings}
\end{theorem}

The inequality $n(m-2) \leq m^2 - 2$ is automatically satisfied for $m = 1$  (Ricci curvature), $m=2$ (bi-Ricci curvature), $m = n-2$, and $m = n-1$ (scalar curvature). Moreover, the inequality $n(m-2) \leq m^2-2$ holds for all $n \leq 7$ and all $1 \leq m \leq n-1$. Surprisingly, in dimension $n \geq 8$, the inequality $n(m-2) \leq m^2 - 2$ fails for $m = 3$ (tri-Ricci curvature) and $m=4$ (tetra-Ricci curvature). 

The second step, which essentially
is given in work of R.~Schoen and S.-T.~Yau
\cite{Schon-Yau:1979:Structure_PositiveScalar},
gives a topological condition for the existence of a stable weighted slicing:

\begin{theorem}[Existence of stable weighted slicings] \ \\ 
Assume $n \leq 7$ and $1 \leq m \leq n-1$. Let $N^n$ be a closed and orientable manifold of dimension $n$,
 and suppose that there exists a closed and orientable manifold 
 $M^{n-m}$ and a map $F: N^n \rightarrow M^{n-m} \times \Torus^m$
 with non-zero degree.
 Then for each Riemannian metric $g$ on $N^n$ there exists a stable  weighted slicing
 \[
  \Sigma_m \subset \Sigma_{m-1} \subset \dots \subset \Sigma_1 \subset \Sigma_0 = N^n
 \] 
of order $m$.  In conjunction with Theorem \ref{theorem:SmoothSlicings} we deduce that the manifold $N$ does not admit a metric with positive $m$-intermediate curvature.
 \label{Theorem: Existence of stable weighted slicings}
\end{theorem}

The dimensional restriction allows us to invoke the regularity theory for hypersurfaces minimizing
a weighted area. 
As a consequence of the above theorem we observe the following corollary:

\begin{corollary}[Nonexistence of metrics of positive $m$-intermediate curvature] \ \\
The product manifolds
$N^n = M^{n-m} \times \Torus^m$ do not admit a metric
of positive $m$-intermediate curvature
for $n \leq 7$ and $1 \leq m \leq n-1$.
\end{corollary}
In particular, the manifold $\Sphere^2 \times \Torus^2$ does not admit a metric of positive bi-Ricci curvature.

In Section 2 we introduce our notation and recall
the first and second variation formula for weighted area.
In Section 3, we describe the proof of Theorem \ref{theorem:SmoothSlicings}. Afterwards, in Section 4, we give the proof of Theorem \ref{Theorem: Existence of stable weighted slicings} and establish existence of stable weighted slicings under topological assumptions.

\textbf{Acknowledgements:} 
The first author was supported by the National Science Foundation under grant DMS-2103573 and by the Simons Foundation. The second author would like to thank Hubert Bray and
Yiyue Zhang for their interest in this work, and he acknowledges the hospitality of Columbia University, where this project was initiated.

\section{The first and second variation of weighted area}
\label{Section:Preliminaries}

For a Riemannian manifold $(N^n,g)$ we consider its
Levi-Civita connection $D$ and its 
Riemann curvature tensor $\Rm_N$ given by the formula
\[
 \Rm_N(X,Y,Z,W)
 =
 -g( D_X D_Y Z - D_Y D_X Z - D_{[X,Y]} Z ,W)
\]
for vector fields $X,Y,Z,W \in \Gamma(TN)$. \\
Consider a two-sided embedded submanifold $\Sigma^{n-1}$.
We denote its induced Levi-Civita connection by $D_{\Sigma}$,
its unit normal vector field by $\nu \in \Gamma(N\Sigma)$,
its scalar-valued second fundamental form by $h_{\Sigma}$
and its mean curvature (the trace of the scalar-valued
second fundamental form over $\Sigma$) by $H_{\Sigma}$.
The gradient of a smooth function on $N$ or $\Sigma$
is denoted by $D_N f$ or $D_{\Sigma} f$.

Our arguments employ the first and second variation formula
of a suitably weighted area:
Consider a Riemannian manifold $(N^n,g)$,
a smooth positive function $\rho: N \rightarrow \R$,
and an embedded two-sided closed manifold
$\Sigma \subset N^n$.
For a given smooth function $f \in C^{\infty}(\Sigma)$
we consider a variation 
$F: (-\epsilon, \epsilon) \times \Sigma \rightarrow N$ with $F(0,x) = x$ and $\left . \frac{\partial}{\partial s} F(s,x) \right |_{s=0} = f(x) \, \nu(x)$. In the following, we denote the map $F(s,\cdot)$ by $F_s$. Moreover, we denote by $\Sigma_s$ the image of $F_s$ and by $\nu_s$ the unit normal vector field to $F_s$. 

By precomposing the maps $F_s$ with suitable tangential diffeomorphisms, we can arrange that the variation is normal in the sense that 
\[\frac{\partial}{\partial s} F_s = f_s \, \nu_s,\] 
where $f_s$ is a smooth function on $\Sigma_s$.

We consider the weighted area defined by
  \[
   \mathcal{H}^{n-1}_{\rho}(\Sigma) := \int_{\Sigma} \rho \, d\mu.
  \]

We recall the classical formulae
for the first and second variation of weighted area:
 
\begin{proposition}[First variation of weighted area] \ \\
  The first variation
  of weighted area is given by
  \begin{align*}
      \left. \frac{d}{ds} 
      \mathcal{H}^{n-1}_{\rho}(\Sigma_s) \right|_{s=0}
   = \int_{\Sigma} \rho f \left( H_{\Sigma} + \langle D_N \log \rho, \nu \rangle \right) \, d\mu.
  \end{align*}
 \end{proposition}

\begin{proof} 
This is a consequence of the first variation formula for area,
and the chain rule.
\end{proof}

\begin{corollary} \ \\ 
 Suppose $\Sigma$ is a critical point
 of weighted area. Then we have
 \[
 H_{\Sigma} = - \langle D_N \log \rho, \nu \rangle.
 \]
 \label{Corollary:FirstVariation_WeightedArea}
\end{corollary}
For a constant weight
we recover the minimal surface equation $H_{\Sigma} = 0$.

  \begin{proposition}[Second variation formula on critical points] \ \\
  If $\Sigma$ is a critical point of the weighted area functional, then the second variation of weighted area is given by  
  \begin{align*}
    &\left. \frac{d^2}{ds^2} \mathcal{H}^{n-1}_{\rho}(\Sigma_s) \right|_{s=0}  \\
   =&
     \int_{\Sigma}
   \rho \left( -f \Delta_{\Sigma} f - 
   \left( |h_{\Sigma}|^2 + \Ric_N(\nu, \nu) \right) f^2
   + f^2 (D_N^2 \log \rho)(\nu, \nu)
   - f
   \langle D_{\Sigma} \log \rho, D_{\Sigma} f \rangle
   \right) 
   \, d\mu.
  \end{align*}
    \label{Proposition:SecondVariationArea-Weight-CriticalPoints}
 \end{proposition}

\begin{proof} 
 We use normal variations for our computation,
 and hence the first derivative is given by
 \[
 \frac{d}{ds} \int_{\Sigma_s} \rho \, d\mu_s
 =
 \int_{\Sigma_s} \rho f_s \left ( H_{\Sigma_s} + \langle D_N \log \rho,\nu_s \rangle \right ) \, d\mu_s. 
\]
We now differentiate both sides of this equation with respect to $s$, and evaluate the result at $s=0$. By the variation formulas for hypersurfaces, compare for example with \cite{HuiskenPolden}, the first order change in the mean curvature is given by
\[
\left . \frac{\partial}{\partial s} H_{\Sigma_s} \right |_{s=0} = -\Delta_\Sigma f -   \left( |h_{\Sigma}|^2 + \Ric_N(\nu, \nu) \right) f,
\] 
whereas the first order change in the normal vector field is given by 
\[
\left . D_s \nu_s \right |_{s=0} = -D_\Sigma f.
\] 
This implies 
\begin{align*} 
&\left . \frac{\partial}{\partial s} \left ( H_{\Sigma_s} + \langle D_N \log \rho,\nu_s \rangle \right ) \right |_{s=0} \\ 
&= -\Delta_\Sigma f -   \left( |h_{\Sigma}|^2 + \Ric_N(\nu, \nu) \right) f + (D_N^2 \log \rho)(\nu,\nu) f - \langle D_\Sigma \log \rho,D_\Sigma f \rangle, 
\end{align*} 
hence 
 \begin{align*}
    &\left. \frac{d^2}{ds^2} \mathcal{H}^{n-1}_{\rho}(\Sigma_s) \right|_{s=0}  \\
   =&
     \int_{\Sigma}
   \rho f \left( -\Delta_{\Sigma} f - 
   \left( |h_{\Sigma}|^2 + \Ric_N(\nu, \nu) \right) f
   + (D_N^2 \log \rho)(\nu, \nu) f
   - 
   \langle D_{\Sigma} \log \rho, D_{\Sigma} f \rangle
   \right) 
   \, d\mu.
  \end{align*}
\end{proof}

 For a constant weight we recover the usual second variation formula for minimal hypersurfaces:
  \begin{align*}
    \left. \frac{d^2}{ds^2} \mathcal{H}^{n-1}(\Sigma_s)  \right|_{s=0} 
   =&
     \int_{\Sigma}
    \left( -f \Delta_{\Sigma} f - 
   \left( |h_{\Sigma}|^2 + \Ric_N(\nu, \nu) \right) f^2
   \right) 
   \, d\mu.
  \end{align*}

\section{Properties of stable weighted slicings}
\label{Section:WeightedSlicings}

Let $(N^n,g)$ be a closed and orientable Riemannian manifold of dimension $\dim N = n$. Throughout this section, we assume that we are given a stable weighted slicing of order $m$. Our goal is to show that the metric $g$ cannot have positive $m$-intermediate curvature.
 
 By the first variation formula for weighted area, Corollary
 \ref{Corollary:FirstVariation_WeightedArea},
 the mean curvature $H_{\Sigma_k}$ of the slice $\Sigma_{k}$ in 
 the manifold $\Sigma_{k-1}$
 satisfies for $1 \leq k \leq m$ the relation 
 \[
  H_{\Sigma_k} = - \langle D_{\Sigma_{k-1}} \log \rho_{k-1}, \nu_k \rangle.
 \]

 By the second variation formula for weighted area 
 (compare Proposition \ref{Proposition:SecondVariationArea-Weight-CriticalPoints})
 we obtain for $1 \leq k \leq m$
 the inequality
 \begin{align*}
  0 \leq& \int_{\Sigma_k}
  \rho_{k-1}
  \left(
  - \psi \Delta_{\Sigma_k} \psi
  - \psi \langle D_{\Sigma_k} \log \rho_{k-1}, D_{\Sigma_k}
  \psi \rangle
  \right) \, d\mu \\
  &- 
  \int_{\Sigma_k}
  \rho_{k-1}
  \left( 
  |h_{\Sigma_k}|^2 + \Ric_{\Sigma_{k-1}}(\nu_k, \nu_k)
  - (D_{\Sigma_{k-1}}^2 \log \rho_{k-1})(\nu_k, \nu_k)
  \right) \psi^2
  \, d\mu
 \end{align*}
 for all $\psi \in C^{\infty}(\Sigma_k)$. By Definition \ref{Definition: stable weighted slicing} we may write $\rho_k = \rho_{k-1} \, v_k$, where $v_k > 0$ is the first eigenfunction of the stability operator for the weighted area functional. The function $v_k$ satisfies
   \begin{align*}
  \lambda_k v_k 
  =&
   - \Delta_{\Sigma_k} v_k 
   - \langle D_{\Sigma_k} \log \rho_{k-1}, D_{\Sigma_k} v_k \rangle
  - \left( 
  |h_{\Sigma_k}|^2 + \Ric_{\Sigma_{k-1}}(\nu_k, \nu_k) \right) v_k \\
  &+ (D_{\Sigma_{k-1}}^2 \log \rho_{k-1})(\nu_k, \nu_k)
   v_k,
 \end{align*}
where $\lambda_k \geq 0$ denotes the first eigenvalue of the stability operator.  

 By setting $w_k = \log v_k$ we record the following equation:
 \begin{equation}
 \begin{aligned}
\lambda_k =& - \Delta_{\Sigma_k} w_k - \langle D_{\Sigma_k} \log \rho_{k-1}, D_{\Sigma_k} w_k \rangle - \left ( 
  |h_{\Sigma_k}|^2 + \Ric_{\Sigma_{k-1}}(\nu_k, \nu_k) \right ) \\
  &+ (D_{\Sigma_{k-1}}^2 \log \rho_{k-1})(\nu_k, \nu_k)
  - |D_{\Sigma_k}  w_k|^2.
  \end{aligned}
  \label{equation:LograrithmicWeights}
\end{equation}
 
 We next record two lemmata connecting the second derivatives on consecutive slices.
 \begin{lemma}[First slicing identity] \ \\
We have for $1 \leq k \leq m$ the identiy
\begin{align*}
 \Delta_{\Sigma_k} \log \rho_{k-1} + (D_{\Sigma_{k-1}}^2 \log \rho_{k-1}) (\nu_k, \nu_k)
 =
 \Delta_{\Sigma_{k-1}} \log \rho_{k-1} + H_{\Sigma_k}^2.
\end{align*}
\label{Lemma:FirstSlicingEquality}
\end{lemma}

\begin{proof}
 The above formula follows by 
 applying
 the formula relating 
the Laplace operator on a submanifold
to the Laplace operator on the ambient space
\begin{align*}
 \Delta_{\Sigma_k} f + (D_{\Sigma_{k-1}}^2 f) (\nu_k, \nu_k)
 =
 \Delta_{\Sigma_{k-1}} f - H_{\Sigma_k} 
 \langle D_{\Sigma_{k-1}} f, \nu_k \rangle.
\end{align*}
to the function $f = \log \rho_{k-1}$.
The gradient term on the right-hand side is rewritten
by using the first variation formula
for weighted area 
 \[
  H_{\Sigma_k} = - \langle D_{\Sigma_{k-1}} \log \rho_{k-1}, \nu_k \rangle.
 \]
from Corollary \ref{Corollary:FirstVariation_WeightedArea}.
\end{proof}
 
 \begin{lemma}[Second slicing identity] \ \\
We have for $1 \leq k \leq m-1$
the identity
\begin{align*}
 \Delta_{\Sigma_{k}} \log \rho_{k}
  =& \Delta_{\Sigma_{k}} \log \rho_{k-1}
 +
   (D_{\Sigma_{k-1}}^2 \log \rho_{k-1})(\nu_k, \nu_k) \\
 &-
 \left(
 \lambda_{k}
 + |h_{\Sigma_{k}}|^2
 + \Ric_{\Sigma_{k-1}}(\nu_k, \nu_k)
 +
\langle D_{\Sigma_{k}} \log \rho_{k}, D_{\Sigma_{k}} w_{k} \rangle
 \right).
\end{align*}
\label{Lemma:SecondSlicingEquality}
\end{lemma}

\begin{proof}
This follows from the identity $\log \rho_k = w_k + \log \rho_{k-1}$ together with the equation \eqref{equation:LograrithmicWeights}.
\end{proof}

\begin{lemma}[Stability inequality on the bottom slice] \ \\
 On the bottom slice $\Sigma_m$
 we have the inequality
 \begin{align*}
  \int_{\Sigma_m}
  \rho_{m-1}^{-1}
  \left(
  \Delta_{\Sigma_{m-1}} \log \rho_{m-1}
  + H_{\Sigma_m}^2 \right) d\mu \geq
\int_{\Sigma_m} \rho_{m-1}^{-1}  \left(
|h_{\Sigma_{m}}|^2 + \Ric_{\Sigma_{m-1}}(\nu_{m}, \nu_{m})
  \right) 
  \, d\mu.
\end{align*}
\label{Lemma:StabilityInequality_BottomSlice}
\end{lemma}

\begin{proof}
By the second variation of weighted area
(compare Proposition \ref{Proposition:SecondVariationArea-Weight-CriticalPoints}) the stability inequality on the bottom slice $\Sigma_m$ gives
\begin{align*}
  0 \leq& \int_{\Sigma_m}
  \rho_{m-1}
  \left(
  - \psi \Delta_{\Sigma_m} \psi
  - \psi \langle D_{\Sigma_m} \log \rho_{m-1}, D_{\Sigma_m}
  \psi \rangle
  \right) \, d\mu \\
&-
\int_{\Sigma_m}
\rho_{m-1}
  \left(|h_{\Sigma_m}|^2 + \Ric_{\Sigma_{m-1}}(\nu_{m}, \nu_{m})
  - (D_{\Sigma_{m-1}}^2 \log \rho_{m-1})(\nu_{m}, \nu_{m})
  \right) \psi^2
  \, d\mu
 \end{align*}
for all $\psi \in C^{\infty}(\Sigma_m)$.
 Since the weight $\rho_{m-1}$ is positive, we may
 use the direction $\psi = \rho_{m-1}^{-1}$
 in the stability inequality, and observe
 \begin{align*}
  -  \Delta_{\Sigma_m} \psi
  &=
  -\Delta_{\Sigma_m} \, \rho_{m-1}^{-1}
  = \rho_{m-1}^{-1} \Delta_{\Sigma_m} \log \rho_{m-1}
  - \rho_{m-1}^{-3} |D_{\Sigma_m} \rho_{m-1}|^2, \\
   - \langle D_{\Sigma_m} \log \rho_{m-1}, D_{\Sigma_m}
  \psi \rangle
  &=
  - 
  \langle D_{\Sigma_m} \log \rho_{m-1}, D_{\Sigma_m}
  \rho_{m-1}^{-1}
  \rangle
  = 
  \rho_{m-1}^{-3} |D_{\Sigma_m} \rho_{m-1}|^2.
 \end{align*}
 The gradient terms in the previous formulae cancel,
 and we obtain by rearrangement
  \begin{align*}
  &\int_{\Sigma_m}
  \rho_{m-1}^{-1}
  \left(
  \Delta_{\Sigma_m} \log \rho_{m-1}
+ (D_{\Sigma_{m-1}}^2 \log \rho_{m-1})(\nu_{m}, \nu_{m}) \right) d\mu\\
\geq &\int_{\Sigma_m} \rho_{m-1}^{-1}  \left(
|h_{\Sigma_{m}}|^2 + \Ric_{\Sigma_{m-1}}(\nu_{m}, \nu_{m})
  \right) 
  \, d\mu.
\end{align*}
 Finally, we use the first slicing equality from Lemma
 \ref{Lemma:FirstSlicingEquality} to replace
 \[
  \Delta_{\Sigma_m} \log \rho_{m-1}
  + (D_{\Sigma_{m-1}}^2 \log \rho_{m-1})(\nu_m, \nu_m)
  =
  \Delta_{\Sigma_{m-1}} \log \rho_{m-1}
  + H_{\Sigma_m}^2.
 \]
\end{proof}

\begin{lemma}[Main inequality]  \ \\
We have the inequality
\begin{align*}
    \int_{\Sigma_m} \rho_{m-1}^{-1}
    \left( 
    \Lambda + \mathcal{R} + \mathcal{E} + \mathcal{G}
    \right) \, d\mu \leq 0,
\end{align*}
where the eigenvalue term $\Lambda$,
the intrinsic curvature term $\mathcal{R}$,
the extrinsic curvature term $\mathcal{E}$,
and the gradient term $\mathcal{G}$ are given by
\begin{align*}
\Lambda
&=\sum_{k=1}^{m-1} \lambda_k, \;
\mathcal{R}
=\sum_{k=1}^m\Ric_{\Sigma_{k-1}}(\nu_k,\nu_k), \;
\mathcal{G}
=\sum_{k=1}^{m-1}
\langle
 D_{\Sigma_k} \log \rho_k, D_{\Sigma_k} w_k
\rangle, \\
\; \text{and} \; \;
\mathcal{E}
&=\sum_{k=1}^m |h_{\Sigma_k}|^2 - \sum_{k=2}^m H_{\Sigma_k}^2.
\end{align*}
\label{Lemma:MainInequality} 
\end{lemma}

\begin{proof}
If we substitute the first slicing equality,
Lemma \ref{Lemma:FirstSlicingEquality},
into the second slicing equality, 
Lemma \ref{Lemma:SecondSlicingEquality},
we obtain for $1 \leq k \leq m-1$
the identity
\begin{align*}
    \Delta_{\Sigma_k} \log \rho_k
 = \Delta_{\Sigma_{k-1}} \log \rho_{k-1}
 +
  H_{\Sigma_k}^2
  &-
 \left(
 \lambda_{k}
 + |h_{\Sigma_{k}}|^2
 + \Ric_{\Sigma_{k-1}}(\nu_k, \nu_k)
 +
\langle D_{\Sigma_{k}} \log \rho_{k}, D_{\Sigma_{k}} w_{k} \rangle
 \right).
\end{align*}

Summation of the above formula over $k$ from $1$ to $m-1$ yields 
\begin{align*}
    \Delta_{\Sigma_{m-1}} \log \rho_{m-1}
 =& \Delta_{\Sigma_0} \log \rho_0
 +\sum_{k=1}^{m-1} H_{\Sigma_k}^2  \\
 &-
 \sum_{k=1}^{m-1}
   \left(
  \lambda_{k}
  + |h_{\Sigma_{k}}|^2
  + \Ric_{\Sigma_{k-1}}(\nu_{k}, \nu_{k})
  + 
 \langle D_{\Sigma_{k}} \log \rho_{k}, D_{\Sigma_{k}} w_{k} \rangle
  \right).
\end{align*}

We plug this equation into the stability inequality,
Lemma \ref{Lemma:StabilityInequality_BottomSlice}.
Moreover, we observe that the weight $\rho_0$ is constant,
the mean curvature of the top slice $H_{\Sigma_1}$ vanishes,
and that the stability inequality contains the mean curvature 
term $H_{\Sigma_m}^2$
the extrinsic curvature term $|h_{\Sigma_m}|^2$
and the curvature term $\Ric_{\Sigma_{m-1}}(\nu_m, \nu_m)$.
Then the lemma follows by grouping the terms suitably.
\end{proof}

We consider two examples to illustrate the structure
of the curvature terms:
\begin{example}[Positive  Ricci curvature and $m=1$] \ \\
In the case $m=1$ we have the slicing $\Sigma_1 \subset \Sigma_0 = N^n$
and we recover the classic result on the instability
of minimal hypersurfaces in positive Ricci curvature $\Ric > 0$.
Indeed, we have $\Lambda = \mathcal{G} = 0$,
$\mathcal{E} = |h_{\Sigma_1}|^2$, and $\mathcal{R} = \Ric_{N}(\nu_1, \nu_1)$.
Thus $\mathcal{R} + \mathcal{E} > 0$.
Combined with the existence theory for stable weighted slicings from 
Section \ref{Section:ExistenceWeightedSlicings}
this implies the non-existence of metrics of positive Ricci curvature
on manifolds with topology $N^n = M^{n-1} \times \Sphere^1$ in dimension
$\dim N \leq 7$.
\end{example}

\begin{example}[Positive bi-Ricci curvature and $m=2$] \ \\
 In the case $m=2$ we have the slicing $\Sigma_2 \subset \Sigma_1 \subset \Sigma_0 = N^n$.
 We moreover observe $\Lambda = \lambda_1 \geq 0$,
 $\mathcal{G} = |D_{\Sigma_1} w_1|^2 \geq H_{\Sigma_2}^2$,
 and the curvature terms $\mathcal{E}$ and $\mathcal{R}$ are given by
\begin{align*}
 &\mathcal{E}
 = |h_{\Sigma_1}|^2 + |h_{\Sigma_2}|^2 - H_{\Sigma_2}^2, \\
 \; \text{and} \; 
 &\mathcal{R}
 =
 \Ric_{N}(\nu_1, \nu_1)
 + \Ric_{N}(\nu_2, \nu_2)
 - \Rm_{N}(\nu_1, \nu_2, \nu_1, \nu_2)
 - (h_{\Sigma_1}^2)(\nu_1, \nu_1).
\end{align*}
Thus if we assume positive bi-Ricci curvature we have
$ \Lambda + \mathcal{R} + \mathcal{E} + \mathcal{G} > 0$.
This shows a non-existence result for stable weighted slicings
of order two.
Combined with the existence theory for stable weighted slicings from Section \ref{Section:ExistenceWeightedSlicings}
this implies that a manifold with topology $N^n = M^{n-2} \times \Torus^2$
 (with $n \leq 7$) does not admit a metric of positive bi-Ricci curvature.
\end{example}

The eigenvalue term $\Lambda$ is non-negative, 
since it is the sum of the non-negative eigenvalues.
We will estimate the other terms below.

The first step is to estimate the gradient terms:

\begin{lemma}[Estimate of gradient terms] \ \\
We have the estimate
\begin{align*}
    \mathcal{G} \geq \sum_{k=2}^m
    \left(
    \frac{1}{2}+\frac{1}{2(k-1)}
    \right)
    H_{\Sigma_k}^2.
\end{align*}
\label{Lemma:Estimate-GradientTerms}
\end{lemma}

\begin{proof}
We define for $k \geq 1$ the nonnegative real numbers $\alpha_k$ by 
\[
 \alpha_k = \frac{k-1}{2k}.
\]
By direct computation one verifies the identity
\begin{align*}
    1 - \alpha_{k-1} = \frac{1}{4 \alpha_k}
\end{align*}
for $k \geq 2$. Using the identity $H_{\Sigma_{k+1}}=-\langle D_{\Sigma_{k}} \log \rho_{k}, \nu_{k+1} \rangle$, we obtain 
\begin{align*}
 &\langle D_{\Sigma_k} \log \rho_k, D_{\Sigma_k} w_k \rangle \\ 
 =&
 \langle D_{\Sigma_k} \log \rho_k,
 D_{\Sigma_k} (\log \rho_k - \log \rho_{k-1}) \rangle \\
 =&
 (1 - \alpha_k) |D_{\Sigma_k} \log \rho_k|^2  - \frac{1}{4 \alpha_k}
 |D_{\Sigma_k} \log \rho_{k-1}|^2 \\ 
& +
 \alpha_k \, \left|
 D_{\Sigma_k} \log \rho_k
 - \frac{1}{2 \alpha_k} D_{\Sigma_k} \log \rho_{k-1}
 \right|^2 \\ 
=& (1 - \alpha_k) \, H_{\Sigma_{k+1}}^2 + (1 - \alpha_k) \, |D_{\Sigma_{k+1}} \log \rho_k|^2 - (1 - \alpha_{k-1}) \, |D_{\Sigma_k} \log \rho_{k-1}|^2  \\ 
&+
\alpha_k \, \left|
 D_{\Sigma_k} \log \rho_k
 - \frac{1}{2 \alpha_k} D_{\Sigma_k} \log \rho_{k-1}
 \right|^2 
\end{align*}
for $2 \leq k \leq m-1$. Summation over $k$ from $2$ to $m-1$ yields the formula 
\[ \sum_{k=2}^{m-1}
 \langle D_{\Sigma_k} \log \rho_k, D_{\Sigma_k} w_k \rangle \geq \sum_{k=2}^{m-1} (1-\alpha_k) \, H_{\Sigma_{k+1}}^2 + (1-\alpha_{m-1}) \, |D_{\Sigma_m} \log \rho_{m-1}|^2 - |D_{\Sigma_2} \log \rho_1|^2.\] 
Moreover, the identity $H_{\Sigma_2} = -\langle D_{\Sigma_1} \log \rho_1,\nu_2 \rangle$ implies 
\[\langle D_{\Sigma_1} \log \rho_1,D_{\Sigma_1} w_1 \rangle = |D_{\Sigma_1} \log \rho_1|^2 = H_{\Sigma_2}^2 + |D_{\Sigma_2} \log \rho_1|^2.\] 
Adding the two inequalities gives 
\[ \sum_{k=1}^{m-1}
 \langle D_{\Sigma_k} \log \rho_k, D_{\Sigma_k} w_k \rangle \geq \sum_{k=1}^{m-1} (1-\alpha_k) \, H_{\Sigma_{k+1}}^2 + (1-\alpha_{m-1}) \, |D_{\Sigma_m} \log \rho_{m-1}|^2.\] 
\end{proof}

In the next step we rewrite the intrinsic curvature
terms with the help of the Gauss equations:

\begin{lemma}[Iterated Gauss equations] \ \\
The curvature term $\mathcal{R}$ is given by
\begin{align*}
     \mathcal{R}
     = 
     \mathcal{C}_{m}(e_1, \dots, e_m)
     + 
  \sum_{k=1}^{m-1} \sum_{p = k+1}^{m} \sum_{q = p+1}^n
  \left(
   h_{\Sigma_k}(e_p, e_p) h_{\Sigma_p}(e_q, e_q)
   -
   h_{\Sigma_k}(e_p, e_q)^2 
   \right),
\end{align*}
where $\mathcal{C}_m$ denotes the $m$-intermediate curvature of the Riemannian manifold $(N^n,g)$.
\label{Lemma:IteratedGaussEquations}
\end{lemma}

\begin{proof}
Fix a point $x \in \Sigma_m$
and consider an orthornomal basis $\{e_1, \dots, e_n\}$ of $T_x N$
with $e_j = \nu_j$ for $1 \leq j \leq m$ as above.
We observe by the definition of the Ricci curvature
on the slice $\Sigma_{k-1}$,
and by the Gauss equations the formula
 \begin{align*}
  \Ric_{\Sigma_{p-1}}(\nu_p, \nu_p)
  &= \Ric_{\Sigma_{p-1}}(e_p, e_p)
  =
  \sum_{q=p+1}^n \Rm_{\Sigma_{p-1}}(e_p, e_q, e_p, e_q) \\
  &=
  \sum_{q=p+1}^n 
  \Rm_{N}(e_p, e_q, e_p, e_q)
  +
  \sum_{q=p+1}^n 
  \sum_{k = 1}^{p-1} 
  \left( 
  h_{\Sigma_k}(e_p, e_p) h_{\Sigma_k}(e_q, e_q)
  - h_{\Sigma_k}(e_p, e_q)^2
  \right).
 \end{align*}
 
 Summation over $p$ from $1$ to $m$
 then implies
 \begin{align*}
  \mathcal{R}
  &=
  \sum_{p=1}^m \Ric_{\Sigma_{p-1}}(\nu_p, \nu_p) \\
  &=
  \sum_{p=1}^m \sum_{q=p+1}^n \Rm_N(e_p, e_q, e_p, e_q)
  +
  \sum_{p=1}^m \sum_{q=p+1}^n \sum_{k=1}^{p-1}
  \left(
   h_{\Sigma_k}(e_p, e_p) h_{\Sigma_k}(e_q, e_q)
   -
   h_{\Sigma_k}(e_p, e_q)^2 \right)\\
   &=
   \mathcal{C}_m(e_1, \dots, e_m)
   +
  \sum_{p=1}^m \sum_{q=p+1}^n \sum_{k=1}^{p-1}
  \left(
   h_{\Sigma_k}(e_p, e_p) h_{\Sigma_k}(e_q, e_q)
   -
   h_{\Sigma_k}(e_p, e_q)^2 
  \right).
 \end{align*}
 If we interchange the order of summation, the assertion follows.
\end{proof}

\begin{remark}[Observation on full slicing] \ \\
In the special case $m = n-1$ the curvature term $\mathcal{R}$ can be rewritten as 
\begin{align*}
     \mathcal{R}
     &= 
     \mathcal{C}_{n-1}(e_1, \dots, e_{n-1})
     + 
  \sum_{k=1}^{n-2} \sum_{p = k+1}^{n-1} \sum_{q = p+1}^n
  \left(
   h_{\Sigma_k}(e_p, e_p) h_{\Sigma_p}(e_q, e_q)
   -
   h_{\Sigma_k}(e_p, e_q)^2 
   \right) \\ 
&= \frac{1}{2} \scal_N + \frac{1}{2} \sum_{k=1}^{n-2} \left( H_{\Sigma_k}^2 - |h_{\Sigma_k}|^2 \right).
\end{align*}
(cf. Remark \ref{remark:ConnectionToOtherConditions} (iv)). Note that the mean curvature of the top slice $\Sigma_1$
vanishes, and that $H_{\Sigma_{n-1}}^2 = |h_{\Sigma_{n-1}}|^2$
since $\Sigma_{n-1}$ is one-dimensional. Therefore, for $m=n-1$ we obtain 
\begin{align*}
 \mathcal{R} + \mathcal{E} + \mathcal{G}
 &=
 \frac{1}{2} \scal_N 
 + \frac{1}{2} \sum_{k=1}^{n-1} |h_{\Sigma_k}|^2 - \frac{1}{2} \sum_{k=1}^{n-1}
 H_{\Sigma_k}^2 + \mathcal{G} \\
 &\geq
 \frac{1}{2} \scal_N + \frac{1}{2} \sum_{k=1}^{n-1} |h_{\Sigma_k}|^2
 + \sum_{k=2}^{n-1} \frac{1}{2(k-1)} H_{\Sigma_k}^2.
\end{align*}
In the last step we have used the estimate for the gradient terms
$\mathcal{G}$ from Lemma \ref{Lemma:Estimate-GradientTerms}.
Hence, we recover a similar result as in the computation
of R.~Schoen and S.-T.~Yau \cite{Schoen-Yau:2017:Structure_PositiveScalar_HigherDimensions}.
\end{remark}

In the next step we need to analyze the contributions
coming from the extrinsic curvature. 
We fix $m \in \{2, \dots, n-1\}$, and we define for $1 \leq k \leq m$ 
the extrinsic curvature terms $\mathcal{V}_k$:
\begin{align*}
 \mathcal{V}_1
 =&
  |h_{\Sigma_1}|^2
  + 
   \sum_{p = 2}^{m} \sum_{q = p+1}^n
  \left(
   h_{\Sigma_1}(e_p, e_p) h_{\Sigma_1}(e_q, e_q)
   -
   h_{\Sigma_1}(e_p, e_q)^2 
  \right)
   ,
  \\
  \mathcal{V}_k =&
  |h_{\Sigma_k}|^2
  -  \left( \frac{1}{2} - \frac{1}{2(k-1)} \right)
 H_{\Sigma_k}^2 \\
 &+
\sum_{p = k+1}^{m} \sum_{q = p+1}^n
  \left(
   h_{\Sigma_k}(e_p, e_p) h_{\Sigma_k}(e_q, e_q)
   -
   h_{\Sigma_k}(e_p, e_q)^2 
  \right)
   \; \text{for} \; 2 \leq k \leq m - 1, \\
   \mathcal{V}_m
   =&
   |h_{\Sigma_m}|^2
    -  \left( \frac{1}{2} - \frac{1}{2(m-1)} \right)
 H_{\Sigma_m}^2.
\end{align*}

By combining Lemma \ref{Lemma:Estimate-GradientTerms}, with Lemma \ref{Lemma:IteratedGaussEquations},
and the above expressions $\mathcal{V}_k$ for the extrinsic curvature terms, we obtain:
\begin{lemma} \ \\
For $2 \leq m \leq n-1$ we have the pointwise estimate
\begin{align*}
 \mathcal{R} + \mathcal{E} + \mathcal{G}
 \geq \mathcal{C}_{m}(e_1, \dots, e_m) 
 + \sum_{k=1}^m \mathcal{V}_k.
 \end{align*}
 \label{Lemma:Reduced-MainInequality}
\end{lemma}

In the following lemmata we estimate
the extrinsic curvature terms $\mathcal{V}_k$. The estimate for $\mathcal{V}_m$ follows
from the trace estimate for symmetric two-tensors. The estimate for $\mathcal{V}_1$ uses minimality of the top slice $\Sigma_1$. The estimate for $\mathcal{V}_k$ with $2 \leq k \leq m-1$ is the most involved.

\begin{lemma}[Extrinsic curvature terms on top slice] \ \\
For $2 \leq m \leq n-1$ we have the estimate
\begin{align*}
    \mathcal{V}_1
    \geq \frac{m^2-2-n(m-2)}{2(n-m)(m-1)} \left ( 
    \sum_{p=2}^m h_{\Sigma_1}(e_p,e_p) \right )^2.
\end{align*}
\label{Lemma:ExtrinsicCurvature_TopSlice}
\end{lemma}

\begin{proof}
To estimate the term $\mathcal{V}_1$, we begin by discarding the off-diagonal terms of the second fundamental form $h_{\Sigma_1}$:
 \begin{align*}
\mathcal{V}_1
 =&
  |h_{\Sigma_1}|^2
  + 
   \sum_{p = 2}^{m} \sum_{q = p+1}^n
  \left(
   h_{\Sigma_1}(e_p, e_p) h_{\Sigma_1}(e_q, e_q)
   -
   h_{\Sigma_1}(e_p, e_q)^2 
  \right) \\
  \ge& 
   \sum_{p = 2}^n
 h_{\Sigma_1}(e_p, e_p)^2
   + 
   \sum_{p = 2}^{m} \sum_{q = p+1}^n
   h_{\Sigma_1}(e_p, e_p) h_{\Sigma_1}(e_q, e_q).
 \end{align*}
The terms on the right hand side can be rewritten as follows: 
 \[
 \mathcal{V}_1
  \ge 
\frac{1}{2} \sum_{p=2}^m h_{\Sigma_1}(e_p,e_p)^2 +  \sum_{q= m+1}^n
 h_{\Sigma_1}(e_q, e_q)^2
   + \sum_{p=2}^m h_{\Sigma_1}(e_p,e_p) \, H_{\Sigma_1} - \frac{1}{2} \left ( \sum_{p=2}^m h_{\Sigma_1}(e_p,e_p) \right )^2.
   \] 
Recall that $H_{\Sigma_1} = 0$. By the Cauchy--Schwarz inequality, 
\[ \sum_{p=2}^m h_{\Sigma_1}(e_p,e_p)^2 \geq \frac{1}{m-1} \left ( \sum_{p=2}^m h_{\Sigma_1}(e_p,e_p) \right )^2\] 
and 
\[ \sum_{q=m+1}^n h_{\Sigma_1}(e_q,e_q)^2 \geq \frac{1}{n-m} \left ( \sum_{q=m+1}^n h_{\Sigma_1}(e_q,e_q) \right )^2 =  \frac{1}{n-m} \left ( \sum_{p=2}^m h_{\Sigma_1}(e_p,e_p) \right )^2,\] 
where in the last step we have used the fact that $H_{\Sigma_1} = 0$. Putting these facts together, the assertion follows.
\end{proof}

\begin{lemma}[Extrinsic curvature terms on intermediate slices] \ \\
For $2 \leq m \leq n-1$ and $2 \leq k \leq m-1$ we have the estimate
\begin{align*}
    \mathcal{V}_k \geq \frac{m^2-2-n(m-2)}{2(m-1)(n-m)}
    \left ( \sum_{q=m+1}^n h_{\Sigma_k}(e_q,e_q) \right )^2.
\end{align*}
\label{Lemma:ExtrinsicCurvature_MidSlices}
\end{lemma}

\begin{proof}
To estimate the term $\mathcal{V}_k$, we start by discarding the off-diagonal terms: 
\begin{align*}
\mathcal{V}_k =&
  |h_{\Sigma_k}|^2
  - \left( \frac{1}{2} - \frac{1}{2(k-1)} \right) H_{\Sigma_k}^2 
+
\sum_{p = k+1}^{m} \sum_{q = p+1}^n
  \left(
   h_{\Sigma_k}(e_p, e_p) h_{\Sigma_k}(e_q, e_q)
   -
   h_{\Sigma_k}(e_p, e_q)^2 
  \right)\\
  \ge&
  \sum_{p=k+1}^nh_{\Sigma_k}(e_p,e_p)^2
  -  \left( \frac{1}{2} - \frac{1}{2(k-1)} \right)
 H_{\Sigma_k}^2 
+
\sum_{p = k+1}^{m} \sum_{q = p+1}^n
   h_{\Sigma_k}(e_p, e_p) h_{\Sigma_k}(e_q, e_q).
\end{align*}
The terms on the right hand side can be rewritten as follows:
\begin{align*}
\mathcal{V}_k \ge& \frac{1}{2} \sum_{p=k+1}^m h_{\Sigma_k}(e_p,e_p)^2 + \sum_{q=m+1}^n h_{\Sigma_k}(e_q,e_q)^2 \\ 
&+ \frac{1}{2(k-1)} \left ( \sum_{p=k+1}^m h_{\Sigma_k}(e_p,e_p) \right )^2 - \left ( \frac{1}{2}-\frac{1}{2(k-1)} \right ) \left ( \sum_{q=m+1}^n h_{\Sigma_k}(e_q,e_q) \right )^2 \\ 
&+ \frac{1}{k-1} \left ( \sum_{p=k+1}^m h_{\Sigma_k}(e_p,e_p) \right ) \left ( \sum_{q=m+1}^n h_{\Sigma_k}(e_q,e_q) \right ).
\end{align*}
The Cauchy--Schwarz inequality gives 
\[\sum_{p=k+1}^m h_{\Sigma_k}(e_p,e_p)^2 \geq \frac{1}{m-k} \left ( \sum_{p=k+1}^m h_{\Sigma_k}(e_p,e_p) \right )^2\] 
and 
\[\sum_{q=m+1}^n h_{\Sigma_k}(e_q,e_q)^2 \geq \frac{1}{n-m} \left ( \sum_{q=m+1}^n h_{\Sigma_k}(e_q,e_q) \right )^2.\] 
Moreover, Young's inequality implies 
\begin{align*} 
\left ( \sum_{p=k+1}^m h_{\Sigma_k}(e_p,e_p) \right ) \left ( \sum_{q=m+1}^n h_{\Sigma_k}(e_q,e_q) \right ) 
\geq & -\frac{m-1}{2(m-k)} \left ( \sum_{p=k+1}^m h_{\Sigma_k}(e_p,e_p) \right )^2 \\ 
&- \frac{m-k}{2(m-1)} \left ( \sum_{q=m+1}^n h_{\Sigma_k}(e_q,e_q) \right )^2. 
\end{align*} 
Putting these facts together, the assertion follows. 
\end{proof}

\begin{lemma}[Extrinsic curvature terms on bottom slice] \ \\
For $2 \leq m \leq n-1$ we have the estimate
\begin{equation}
    \mathcal{V}_m
    \geq
 \frac{m^2 - 2 - n(m-2)}{2(n-m)(m-1)} \, 
 H_{\Sigma_m}^2.
\end{equation}
\label{Lemma:ExtrinsicCurvature_BottomSlice}
\end{lemma}

\begin{proof} 
We observe by the the trace estimate
for symmetric two-tensors the inequality
\begin{align*}
 \mathcal{V}_m
   &=
   |h_{\Sigma_m}|^2
    -  \left( \frac{1}{2} - \frac{1}{2(m-1)} \right)
 H_{\Sigma_m}^2 
 \geq
 \left(
  \frac{1}{n-m} 
  -
  \left(
   \frac{1}{2} - \frac{1}{2(m-1)}
  \right)
 \right)
 H_{\Sigma_m}^2 \\
 &=
 \frac{m^2 - 2 - n(m-2)}{2(n-m)(m-1)} \,
 H_{\Sigma_m}^2.
\end{align*}
\end{proof}

With the above observations we prove our first theorem:

\begin{proof}[Proof of Theorem \ref{theorem:SmoothSlicings}] \ 
Assume that $1 \leq m \leq n-1$ and $n(m-2) \leq m^2-2$. Suppose that $(N^n,g)$ is a closed and orientable
Riemannian manifold which admits a stable weighted slicing
\[
 \Sigma_{m} \subset \Sigma_{m-1} \subset \dots \subset \Sigma_1
 \subset \Sigma_0 = N^n. 
\]
If $m=1$, the stability inequality implies that $(N^n,g)$ cannot have positive Ricci curvature. Hence, it remains to consider the case when $2 \leq m \leq n-1$ and $n(m-2) \leq m^2-2$. In this case, it follows from Lemma \ref{Lemma:ExtrinsicCurvature_TopSlice}, Lemma \ref{Lemma:ExtrinsicCurvature_MidSlices}, and Lemma \ref{Lemma:ExtrinsicCurvature_BottomSlice} that $\mathcal{V}_k \geq 0$ for all $1 \leq k \leq m$. Using Lemma \ref{Lemma:Reduced-MainInequality}, we obtain the pointwise inequality  
\begin{align*}
    \mathcal{R} + \mathcal{E} + \mathcal{G} \geq 
    \mathcal{C}_m(e_1, \dots, e_m).
\end{align*}
If $\mathcal{C}_m(e_1, \dots, e_m)$ is strictly positive, this contradicts our main inequality, Lemma \ref{Lemma:MainInequality}. Therefore, the Riemannian manifold $(N^n,g)$
cannot have positive $m$-intermediate curvature.
\end{proof}

\section{Existence of stable weighted slicings}
\label{Section:ExistenceWeightedSlicings}

In this section we prove existence of stable weighted slicings of order $m$.
The argument uses the mapping degree and is essentially contained in Theorem 4.5
of \cite{Schoen-Yau:2017:Structure_PositiveScalar_HigherDimensions}.
Alternatively, one could also use an argument based on homology,
compare with Theorem 4.6 in \cite{Schoen-Yau:2017:Structure_PositiveScalar_HigherDimensions}.

\begin{proof}[Proof of Theorem \ref{Theorem: Existence of stable weighted slicings}]
Suppose $N^n$ and $M^{n-m}$ are closed and orientable manifolds,
and suppose $F: N^n \rightarrow \Torus^m \times M^{n-m}$ is a map of degree $d \neq 0$.
 The projection of $F$ onto the factors
 yields maps $f_0: N \rightarrow M$
 and maps $f_1, \dots, f_m: N \rightarrow \Sphere^1$.
 Let $\Theta$ be a top-dimensional form of
 the manifold $M$
 normalized such that $\int_M \Theta = 1$,
 and let $\theta$ be a one-form
 on the circle $\Sphere^1$ with $\int_{\Sphere^1}\theta = 1$.
We define the pull-back forms $\Omega := f_0^* \, \Theta$
and $\omega_j := f_j^* \, \theta$.
By the normalization condition we deduce that
$\int_N \omega_1 \wedge \dots \wedge \omega_m \wedge \Omega = d$.

We claim that one can construct closed and orientable
slices $\Sigma_k$
and weights $\rho_k$ such that
$\int_{\Sigma_k} \omega_{k+1} \wedge \dots \wedge \omega_m \wedge \Omega = d$.
We prove the claim by induction. The base case $k=0$
holds by the previous observation and by setting $\Sigma_0 := N$ and $\rho_0 := 1$.
For the induction step we suppose that we have constructed
the slice $\Sigma_{k-1}$
and the weight $\rho_{k-1}$,
such that 
$\int_{\Sigma_{k-1}} \omega_k \wedge \dots \wedge \omega_m \wedge \Omega = d$.

We define a class $\mathcal{A}_k$ by 
\begin{align*}
 \mathcal{A}_{k}
 =
 \left\{
  \Sigma \; \text{is an} \; (n-k)-\text{integer rectifiable current in } \Sigma_k
  \; \text{with} \;
  \int_{\Sigma} \omega_{k+1} \wedge \dots \wedge \omega_{m} \wedge \Omega = d
 \right\}.
\end{align*}

The first step is to show that the
class $\mathcal{A}_k$ is non-empty.
To prove this, let us fix a regular value
$p_k \in \Sphere^1$
of the map $f_k|_{\Sigma_{k-1}}: \Sigma_{k-1} \rightarrow \Sphere^1$. The existence of a regular value follows from Sard's Theorem.

On the complement $\Sphere^1 \backslash \{p_k\}$
the one-form $\theta$ is exact. In other words, there exists a smooth function $\psi_k: \Sphere^1 \backslash \{p_k\} \rightarrow \R$,
such that $d \psi_k = \theta$. Moreover, due to the normalization condition $\int_{\Sphere^1} \theta = 1$, the function
$\psi_k$ jumps by $1$ at $p_k$.

 We next consider the pre-image
$\tilde{\Sigma}_k = \{ x \in \Sigma_{k-1} : f_k(x) = p_k\}$. Since $p_k \in \Sphere^1$ is a regular value of the map $f_k|_{\Sigma_{k-1}}: \Sigma_{k-1} \rightarrow \Sphere^1$, it follows that $\tilde{\Sigma}_k$ is a closed and orientable submanifold of $\Sigma_{k-1}$. 
We define a function $\varphi_k: \Sigma_{k-1} \backslash \tilde{\Sigma}_k \rightarrow \R$ by setting $\varphi_k := \psi_k \circ f_k$. Since the pull-back commutes with the differential, we deduce $d \varphi_k = f_k^* \, (d \psi_k) = f_k^* \, \theta = \omega_k$ on $\Sigma_{k-1} \backslash \tilde{\Sigma}_k$.

The above observation (and the closedness
of the forms $\omega_k, \dots, \omega_m, \Omega$) implies
\begin{equation}
\omega_k \wedge \omega_{k+1} \wedge \dots \wedge \omega_m \wedge \Omega = d (\varphi_k \, \omega_{k+1} \wedge \dots \wedge \omega_m \wedge \Omega).
 \label{Eq:ExistenceTheory_DifferentialForms}
\end{equation}

We first consider the case when $\tilde{\Sigma}_k$ is empty. Integrating the identity
\eqref{Eq:ExistenceTheory_DifferentialForms}
over $\Sigma_{k-1}$ gives 
\begin{align*}
d = \int_{\Sigma_{k-1}} \omega_k \wedge \omega_{k+1} \wedge \dots \wedge \omega_m \wedge \Omega = \int_{\Sigma_{k-1}} d( \varphi_k \, \omega_{k+1} \wedge \cdots \wedge \omega_m \wedge \Omega) = 0.
\end{align*}
This is a contradiction.

It remains to consider the case when $\tilde{\Sigma}_k$ is non-empty. In this case $\tilde{\Sigma}_k$
is a smooth, orientable and embedded hypersurface
in $\Sigma_{k-1}$.
We integrate identity \eqref{Eq:ExistenceTheory_DifferentialForms}
over
$\Sigma_{k-1} \setminus \tilde{\Sigma}_k$.
By Stokes theorem, the integral of the right hand side yields two boundary integrals over $\tilde{\Sigma}_k$. Since the function $\varphi_k$ jumps by $1$ along $\tilde{\Sigma}_k$, we obtain 
\begin{align*} 
d = \int_{\Sigma_{k-1} \setminus \tilde{\Sigma}_k} \omega_k \wedge \omega_{k+1} \wedge \dots \wedge \omega_m \wedge \Omega 
&= \int_{\Sigma_{k-1} \setminus \tilde{\Sigma}_k} d (\varphi_k \, \omega_{k+1} \wedge \dots \wedge \omega_m \wedge \Omega) \\ 
&= \pm \int_{\tilde{\Sigma}_k} \omega_{k+1} \wedge \dots \wedge \omega_m \wedge \Omega, 
\end{align*}
where the sign depends on the choice of orientation of $\tilde{\Sigma}_k$. Therefore, we can make a choice of orientation so that $\tilde{\Sigma}_k$ belongs to the class $\mathcal{A}_k$. In particular, the class $\mathcal{A}_k$ is non-empty.

We consider the variational problem
\begin{align*}
 \sigma_{k} = \inf \left\{ \mathbb{M}_{\rho_{k-1}, n-k}(\Sigma) : \Sigma \in \mathcal{A}_{k} \right\},
\end{align*}
where $\mathbb{M}_{\rho_{k-1}, n-k}$ denotes the $\rho_{k-1}$-weighted mass functional
on $(n-k)$-integer rectifiable currents.
By the compactness theory for integer rectifiable currents, compare for example Theorem 7.5.3 in \cite{Simon:1983:GMT},
we deduce that there exists an $(n-k)$-integer rectifiable
current $\Sigma_{k+1}$ with mass $\mathbb{M}_{\rho_{k-1}, n-k}(\Sigma_{k}) = \sigma_{k}$.

By the regularity theory
for integer rectifiable currents, compare for example Theorem 7.5.8 in \cite{Simon:1983:GMT} or the survey \cite{DeLellis:GMT},
and the dimension bound $n \leq 7$ we deduce
that $\Sigma_{k}$ is a smooth and orientable
(and hence two-sided) hypersurface.
Moreover, the smooth surface $\Sigma_k$ is stable with respect to variations
of the weighted area, and therefore we can find a positive
first eigenfunction $v_k$
of the weighted stability operator.
Defining the weight $\rho_k$ by the formula $\rho_k = \rho_{k-1} \cdot v_k$ completes the induction step.
 \end{proof}


\begin{thebibliography}{99}
\bibitem{BSW}
M. Burkemper, C. Searle and M. Walsh, \textit{Positive $(p, n) $-intermediate scalar curvature and cobordism}, arxiv:2110.12069 

\bibitem{ChodoshLi}
O. Chodosh and C. Li, \textit{Generalized soap bubbles and the topology of manifolds with positive scalar curvature}, arxiv:2008.11888 

\bibitem{ChodoshLiSurvey}
O. Chodosh and C. Li, \textit{Recent results concerning topological obstructions to positive scalar curvature}, Perspectives in Scalar Curvature, World Scientific Publishing Company vol. 2 (2022)

\bibitem{DeLellis:GMT}
C.~De~Lellis, \textit{The regularity theory for the area functional (in geometric measure theory)}, International Congress of Mathematicians, 2022

\bibitem{Gromov}
M. Gromov, \textit{No metrics with positive scalar curvatures on aspherical 5-manifolds}, arxiv:2009.05332 

\bibitem{GromovLawson}
M. Gromov and H. B. Lawson, \textit{Positive scalar curvature and the Dirac operator on complete Riemannian manifolds}, Inst. Hautes Etudes Sci. Publ. Math. (1983), no. 58, 83--196 (1984)
 
\bibitem{Hoelzel}
S. Hoelzel, \textit{Surgery stable curvature conditions}, Math. Ann., \textbf{365} (2016), no. 1-2, 13--47.

\bibitem{HuiskenPolden}
G. Huisken and A. Polden, \textit{Geometric evolution equations for hypersurfaces}, Calculus of variations and geometric evolution problems (1999): 45-84

\bibitem{Labbi:1997:pCurvature}
M.-L.~Labbi, \textit{Stability of the {{\(p\)}}-curvature positivity under surgeries and manifolds with positive {Einstein} tensor},
Ann. Global Anal. Geom. 15 (1997), 299--312
 
 
\bibitem{Schon-Yau:1979:Structure_PositiveScalar}
R.~Schoen and S.-T.~Yau, \textit{On the structure
of manifolds with positive scalar curvature},
Manuscr. Math. 28 (1979), 159--183
 
\bibitem{Schoen-Yau:2017:Structure_PositiveScalar_HigherDimensions}
R.~Schoen and S.-T.~Yau, \textit{Positive Scalar Curvature
and Minimal Hypersurface Singularities}, arXiv:1704.05490 
 
\bibitem{Shen-Ye:1996:PositiveBiRicci}
Y.~Shen and R.~Ye, \textit{On stable minimal surfaces in manifolds of positive bi-{Ricci} curvatures}, Duke Math. J. 85 (1996), 109--116
 
 \bibitem{Shen-Ye:1997:PositiveBiRicci}
 Y.~Shen and R.~Ye, \textit{On the geometry and topology of manifolds of positive bi-Ricci curvature}, arXiv preprint dg-ga/9708014 (1997)
 
\bibitem{Simon:1983:GMT}
L.~Simon, \textit{Lectures on geometric measure theory}, The Australian National University, Mathematical Sciences Institute, Centre for Mathematics \& its Applications, 1983
\end{thebibliography}
\end{document}